\documentclass[psamsfonts]{amsart}
\usepackage{blindtext}
\usepackage{graphicx}

\usepackage[utf8]{inputenc}
\usepackage{amssymb,amsfonts,amsmath}
\usepackage[all,arc]{xy}
\usepackage{enumerate}
\usepackage{mathrsfs}
\usepackage[letterpaper, portrait, margin=1.4in]{geometry}
\usepackage{amsthm}
\theoremstyle{plain} 
\newtheorem{thm}{Theorem}[section]

\newtheorem{prop}[thm]{Proposition}

\theoremstyle{definition}

\newtheorem{defn}[thm]{Definition}

\newtheorem{con}[thm]{Construction}
\newtheorem{exmp}[thm]{Example}
\newtheorem{nexmp}[thm]{Non-Example}

\newtheorem{notn}[thm]{Notation}

\newtheorem{rem}[thm]{Remark}
\theoremstyle{remark}

\title{Persistence and Sheaves}
\author{Karthik Yegnesh}

\begin{document}

\maketitle
\begin{abstract}
In this note, we descibe a mild generalization of Carlsson and Zomorodian's \textit{persistent homology} of filtered topological spaces, namely \textit{persistent sheaf cohomology}. Given a sheaf of abelian groups on a filtered topological space, we obtain a global description of the sheaf cohomology classes present across the space by studying the persistence of the cohomology classes. As an application, we study the persistent cellular sheaf cohomology of network coding sheaves developed by Ghrist and Hiraoka in [2]. The persistence of network coding sheaf cohomology classes across a filtered digraph (which represents network deterioration phenomena) provides insight into the stability of certain information flows across the network.
\end{abstract}

\section{Introduction}
Let $X$ be a topological space and $\phi:X\rightarrow\mathbb{R}$ a continuous real-valued function. Traditional persistent homology theory seeks to obtain a global description of the homological properties of $X$ via examining the singular homology of subspaces of $X$ induced by $\phi$. That is, persistent homology studies the ``persistence" of the homology groups of each subspace of the filtered space $\phi^{-1}(-\infty,i_0]=X_{0}\hookrightarrow\phi^{-1}(-\infty,i_1]=X_{1}\hookrightarrow\ldots\hookrightarrow\phi^{-1}(-\infty,i_n]=X_{n}=X$, where $i_j<i_k$ for $j<k$ and $n<\infty$. The persistence of homology classes in the filtration yields useful information regarding their significance in the global picture of $X$. In this paper, we study the data of an abelian sheaf on a filtered topological space via studying the persistence of the sheaf cohomology functor restricted to sub-spaces in the filtration. As an application, we study the persistent cellular sheaf cohomology of a network coding sheaves developed in [2]. The persistence of network coding sheaf cohomology classes provide insight into the stability of certain information flows across the network.

\section{Background}
We will recall some basic definitions regarding persistent homology and (co)sheaf (co)homology. We will assume some background with basic algebraic topology, including singular homology and some category theory. For more background, the read is encouraged to look at [1] and [3]. 

\subsection{Persistent Homology}
Persistent homology is a tool which one uses to study the birth and death of topological features in a filtered space. 
\begin{defn}
Let $X_{\bullet}$ be a filtered topological space, i.e a space $X$ equipped with a sequence of nested subspaces $X_{0}\subset X_{1}\subset\ldots\subset X_{n}=X$. Let $H_n(X)$ denote the $n^{th}$ singular homology (with coefficients in a field $k$) vector space of $X$. Fix indices $i$ and $j$ with $i\leq j$. The $(i,j)$-persistent $n^{th}$ homology vector space $H^{i,j}_{n}(X_{\bullet})$ is defined as $H^{i,j}_{n}(X_{\bullet})=\mathrm{im}(H_n(X_i)\rightarrow H_n(X_j))$, where $H_n(X_i)\rightarrow H_n(X_j)$ is induced by the inclusion $X_i\subset X_j$.
\end{defn}
\begin{rem}
The dimension of the $k$-vector space $H^{i,j}_{n}(X_{\bullet})$ is the number of $n$-dimensional holes present in subspace $X_{i}$ that are also present in $X_{j}$.
\end{rem}
\begin{exmp}
If $i=j$, then it is clear that $H^{i,j}_{n}(X_{\bullet})=H_{n}(X_{i})$ since the $k$-linear map induced by $\mathrm{id}:X_{i}\rightarrow X_{i}$ must be the identity map on $H_{n}(X_{i})$.
\end{exmp}

\subsection{Sheaves and Cech cohomology}
We recall some relevant facts about sheaves and cohomology. In this paper, we will use both sheaves and cosheaves, but the information this subsection is easily dualized for the case of cosheaves. An excellent survey of cosheaf theory can be found in Justin Curry's thesis [3].\\

Basically, a sheaf is an association of ``data" (what exactly that means depends on the situation) to open sets of a topological space that is compatible with inclusions $U\hookrightarrow V$ of open sets of the space. Formally:
\begin{defn}
Let $X$ be a topological space and $C$ an abelian category (the reader can safely imagine $C$ to be $\mathrm{Ab}$, the category of abelian groups). A \textit{$C$-valued presheaf} $F$ on $X$ is a contravariant functor $F\colon \mathrm{Open}(X)^{op} \to C$ from the category of open subsets of $X$ to $J$. If $U\subset X$, an element $x\in F(U)$ is a \textit{section} of $F$ over $U$. For a pair of embedded open subsets $V \subset U \subset X$, the induced map on the inclusion $F(U)\to F(V)$ is called the \textit{restriction map.} A presheaf $F$ on $X$ is a \textit{sheaf} if for any open $U \subset X$ and any open cover $\{U_i\}$ of $U$, the sequence $0\rightarrow F(U)\rightarrow\bigoplus_iF(U_i)\rightarrow\bigoplus_iF(\bigcap_iU_i)$ is exact. Note that if $F$ is a sheaf, then $F(\emptyset)=0$, where $0$ denotes the zero object of $C$ (e.g the trivial abelian group, the trivial $k$-vector space, etc.).
\end{defn}

\begin{exmp}
The presheaf which associates to each open set $U\subset X$ its singular $n^{th}$ cohomology $H^{n}(U;G)$ with coefficients in some abelian group $G$ is a sheaf. The sheaf condition in this case is satisfied because the functor $H^{n}$ satisfies the Mayer-Vietoris property.
\end{exmp}
\begin{nexmp}
The constant presheaf $F:\mathrm{Open}(X)\rightarrow\mathrm{Ab}$ which sends each open set to the abelian group $\mathbb{Z}$ and each morphism $U\rightarrow V$ to $\mathrm{id}_{\mathbb{Z}}$ is not a sheaf.
\end{nexmp}
\begin{defn}
Let $X$ be a topological space and $\mathscr{U}=\{U_i\}$ an open cover of $X$, and $F$ a presheaf of abelian groups on $X$. The group of \textit{\v{C}ech $k$-chains} associated to $\mathscr{U}$ is the group $C_k(\mathscr{U},F)= \bigoplus_iF(U_{0,1,\ldots ,k})$, where $U_{0,1,\ldots,k}=\bigcap_{i=0}^kU_i$.

Equipped with differentials $\partial^k \colon C_k(\mathscr{U};F)\to C_{k+1}(\mathscr{U};F)$, we obtain a \textit{\v{C}ech complex} $C_{\bullet}(\mathscr{U};F)= 0 \to C_0(\mathscr{U};F)\to C_{1}(\mathscr{U};F)\to \ldots \to C_k(\mathscr{U};F)\to \ldots$. We denote the $k^{th}$ \v{C}ech cohomology group associated to $F$ and covering $\mathscr{U}$ by $\check{H}^n(\mathscr{U};F)=\mathrm{ker}(\partial^{k})/\mathrm{im}(\partial^{k-1})$.
\end{defn}
\begin{rem}
\v{C}ech cohomology in general does not coincide with sheaf cohomology (defined via derived functors), but for our purposes and eventual applications, \v{C}ech cohomology suffices.
\end{rem}

\subsection{Cellular Sheaves}
In order to make (co)sheaf (co)homology computable in practical scenarios, one often restrict attention to sheaves over \textit{cell complexes} which are valued in the category of vector spaces over a (usually finite) field. In this paper, we will develop out techniques in the generality of arbritrary topological spaces. However, our main application of studying the persistence of network coding sheaf cohomology classes will involve cellular language.
\begin{defn}
Let $X$ be a cell complex (see [3]). Let $\mathrm{Cell}(X)$ denote the poset of cells in $X$, regarded as a category in which there exists a single arrow $x\rightarrow y$ if and only if $x$ is a face of $y$. A \textit{cellular sheaf} $F$ on $X$ is a functor $F:\mathrm{Cell}(X)\rightarrow\mathrm{Vect}_{k}$, where $\mathrm{Vect}_{k}$ denotes the category of vector spaces over the field $k$. 
\end{defn}
\begin{defn}
Given a cellular sheaf $F$ on $X$, one can define \textit{cellular sheaf cohomology}. It is a computationally feasible adaptation of the Cech/Sheaf cohomology of sheaves on general topological spaces or Grothendieck sites. $\mathrm{Cell}(X)_{k}$ denote the set of $k$-dimensional cells of $X$. Write $x\leq y$ if $x$ is a face of $y$. For $\lambda$ and $\theta$ two cells in $X$, denote by $[\lambda,\theta]$ the signed incidence relation [3, Definition 6.1.9]. Let $C(X;F)_k=\bigoplus_{c\in\mathrm{Cell}(X)_{k}}F(c).$ Define the differential $\partial^k:C(X;F)_k\rightarrow C(X;F)_{k+1}$ by $\partial^k(c)=\sum_{\theta\leq c}[\theta:c]\alpha_{c,\theta}$ for $c\in C(X;F)_k$ and $\alpha_{c,\theta}$ being the restriction. Since $\partial^{2}=0$, define the $k^{th}$ cellular sheaf cohomolgy $k$-vector space $H^{k}_{c}(X;F)=\mathrm{ker}(\partial^{k})/\mathrm{im}(\partial^{k-1})$.
\end{defn}
Cellular sheaf cohomology (particularly in the contex of network coding sheaves [1]) will be used as an application of our tools.

\section{Persistent Sheaf Cohomology}
We can adjust the definition of persistent homology slightly to obtain \textit{persistent sheaf cohomology}. We will place the restriction that the spaces in consideration are cell complexes, so we are dealing with cellular sheaf cohomology.
\begin{notn}
Let $X$ be a finite topological space. Denote by $X_{\bullet}$ a filtration $X_{0}\subset X_{1}\subset\ldots\subset X_{n}=X$. We will assume that the $X_i$ are indexed over an ordered set. Each $X_{i}$ is endowed with the subspace topology induced by the inclusion $i:X_{i}\hookrightarrow X$. Let $F:X^{op}\rightarrow\mathrm{Vect}_{k}$ be a $\mathrm{Vect}_{k}$-valued sheaf over $X$ (which restricts to a sheaf over all the $X_i$). Fix a covering $\mathscr{U}$ on $X$. This restricts to coverings $\mathscr{U}_{i}$ on each of the $X_i$. Let $C_k(\mathscr{U},F)= \bigoplus_iF(U_{0,1,\ldots ,k})$ (see Definition 2.4). Let $\Omega^{i}$ denote the inclusion $\Omega^{i}:X_i\hookrightarrow X_{i+1}$.
\end{notn}
We obtain the following commutative diagram, called the \textit{sheaf persistence complex}. 
\begin{con}
\begin{displaymath}
\xymatrix{ 
\vdots \ar[d] & \vdots \ar[d]  & \vdots \ar[d]     \\
0\to \ldots \to \mathrm{C}_{k-1}^k(\mathscr{U}_{i-1};F) \ar[d]^{\Omega^{i-1}} \ar[r]^{\hspace{10mm}\partial_{k-1}^{i-1}}& \mathrm{C}_{k}^k(\mathscr{U}_{i-1};F)  \ar[d]^{\Omega^{i-1}} \ar[r]^{\hspace{-10mm}\partial_{k}^{i-1}}  & \mathrm{C}_{k+1}^k(\mathscr{U}_{i-1};F)\to \ldots \to 0  \ar[d]^{\Omega^{i-1}} \\
0\to \ldots \to \mathrm{C}_{k-1}^k(\mathscr{U}_i;F) \ar[d]^{\Omega^{i}} \ar[r]^{\hspace{10mm}\partial_{k-1}^{i}}& \mathrm{C}_{k}^k(\mathscr{U}_i;F) \ar[d]^{\Omega^{i}}\ar[r]^{\hspace{-10mm}\partial_{k}^{i}} & \mathrm{C}_{k+1}^k(\mathscr{U}_i;F)\to\ldots \to 0 \ar[d]^{\Omega^{i}} \\
0\to \ldots \to \mathrm{C}_{k-1}^k(\mathscr{U}_{i+1};F)  \ar[d] \ar[r]^{\hspace{10mm}\partial_{k-1}^{i+1}}& \mathrm{C}_{k}^k(\mathscr{U}_{i+1};F) \ar[d]\ar[r]^{\hspace{-10mm}\partial_{k}^{i+1}} & \mathrm{C}_{k+1}^k(\mathscr{U}_{i+1};F)\to \ldots \to 0 \ar[d] \\
\vdots  & \vdots  & \vdots  \\ 
 }
 \end{displaymath}
\end{con}
The cohomology vector spaces of this complex will be used in the definition of \textit{co-persistent sheaf cohomology}, which we now provide. 
\begin{defn}
Let $H^k(X_i;F)$ denote the $k^{th}$ sheaf cohomology vector space of $F$ restricted to $X_i$, i.e $H^k(X;F)=\mathrm{ker}(\partial^{i}_{k})/\mathrm{im}(\partial^{i}_{k-1})$ as dictated by the above diagram. Let $X_{\bullet}$ be a filtered topological space and $F$ a cellular cosheaf on it. The $k^{th}$ $(i,j)$ co-persistent sheaf cohomology vector space $H^k_{i,j}(X;F)$ is defined as $\mathrm{im}(H^k(X_j;F)\rightarrow H^k(X_i;F))$, where $H^k(X_j;F)\rightarrow H^k(X_i;F)$ is induced by the inclusion $X_i\hookrightarrow X_j$.
\end{defn}
Notice that instead of defining the persistent cosheaf homology, we choose to define co-persistent sheaf cohomology. The vector space $H^k_{i,j}$ encodes the Cech cohomology classes of the restriction $F|_{X_{j}}$ that are not destroyed when passing to the subspace $X_i$. The analogue of this in the singular homology world would be studying which homology classes are present in a subspace that are not disrupted when passing to a smaller subspace.

\begin{rem}
Let $\widetilde{F}$ be the (sheafification of) the constant $k$-valued presheaf $F:\mathrm{Open}(X)\rightarrow\mathrm{Vect}_{k}$. Then there is an isomorphism of $k$-vector spaces $H^k_{i,j}(X;F)\simeq\mathrm{im}(H^{k}(X_{j},k)\rightarrow H^{k}(X_{i},k))$. In other words, we can obtain ``persistent singular cohomology" as a special case of persistent sheaf cohomology in the same way that singular cohomology is a special case of sheaf cohomology.
\end{rem}

\begin{rem}
This is very easily dualizable to obtain \textit{persistent cosheaf homology}. 
\end{rem}

In the context of this paper, co-persistent cosheaf cohomology enables us to study which NC sheaf cohomology classes persist through a deteriorating network. We will switch to ``persistent cellular sheaf cohomology," but this is defined in the exact same way as the more general case. 


\subsection{Sheaf Cohomological Persistence Modules and Diagrams}
The theory of quiver representations and persistence diagrams plays a large role in the development of persistent homology in the sense of Carlsson and Zomorodian. We will describe a similar scenario on the context of persistent (co)sheaf homology. We recall a definition first.
\begin{defn}
A \textit{persistence module} of length $n$ is a sequence of vector spaces $V_i$ over a field $k$ indexed over $\{i\in\mathbb{N}|i\leq n\}$ equipped with $k$ linear maps $\varphi_i:V_i\rightarrow V_{i+1}$. Equivalently, this is a functor from the small category $\bullet\rightarrow\ldots\rightarrow\bullet$ (with $n$ objects) to $\mathrm{Vect}_{k}$.
\end{defn}
It is a classical result that every persistence module admits a unique decomposition into direct sums of so called \textit{interval modules}. An interval module is a persistence module of the form $0\rightarrow\ldots\rightarrow k\rightarrow\ldots\rightarrow k\rightarrow0$. 

In persistent homology, one naturally obtains a persistence module by applying the (singular) homology functor $H_{n}:\mathrm{TopSpaces}\rightarrow\mathrm{Vect}_{k}$ to a filtered topological space $X_{\bullet}=X_{0}\hookrightarrow X_{1}\hookrightarrow\ldots$. The lengths of the intervals in the canonical interval decomposition (which represent the lifetimes of homology classes) are recorded in a multiset called a \textit{persistence diagram}.

\begin{defn}
Let $X_{\bullet}$ be a filtered topological space.
Let $\mathbb{Z}_{\infty}$ denote the set $\mathbb{Z}\cup\{\infty\}$. A degree $n$ persistence diagram is roughly a multiset over $\mathbb{Z}_{\infty}\times\mathbb{Z}_{\infty}$, where the multiplicity $\mu(i,j)$ of a point $(i,j)\in\mathbb{Z}_{\infty}$ is $\mathrm{dim}(H_{n}^{i,j}(X_{\bullet}))$. The points parametrize the lifetimes of homology classes in $X_{\bullet}$. The points of a persistence diagram therefore correspond to "intervals" in the interval decomposition of the persistence module obtain via the homology of $X_{\bullet}$. The multiplicity function records the number of intervals of a particular length and position exist.
\end{defn}

The notion of a persistence diagram thus can be generalized to any situation in which there is a suitable interpretation of a persistence module. 

\begin{con}
Let $F$ be a sheaf of vector spaces on filtered topological space $X_{\bullet}=X_{0}\hookrightarrow X_{1}\hookrightarrow\ldots$. Let $\mathscr{U}$ be a covering of $X$ and denote by $\mathscr{U}_{i}$ the restriction of $\mathscr{U}$ to $X_{i}$. By applying the Cech cohomology functor $H^{k}(\mathscr{U}_{-};F)$ to each $X_{i}$, we obtain a persistence module $H^{k}(\mathscr{U}_{0};F)\rightarrow H^{k}(\mathscr{U}_{1};F)\rightarrow\ldots$. This admits a direct-sum decomposition into interval modules. Let $\mathrm{Int}_{k}(i,j)$ (for $i\leq j\leq\infty$) denote the set of interval modules of length $j-i$ such that the first non-trivial vector space in each of the interval modules is at position $i$ and last at position $j$. 
\end{con}
\begin{rem}
The interval modules in the last construction represent the lifetimes of sheaf cohomology classes in the $X_{i}$ as one passes to increasingly smaller subspaces of $X$. The long intervals describe sheaf cohomology classes which ``persist" through the ``deteriorating space," while the short intervals indicate that certain classes do not.
\end{rem}
\begin{rem}
If instead we use the cosheaf homology functor where the cosheaf is the constant functor taking values in the field $k$, then the persistence module obtained is precisely the one obtained by taking persistent homology with coefficients in $k$.
\end{rem}
We now define \textit{degree $n$ sheaf persistence diagrams} based on the persistence modules associated to a sheaf on a filtered space described previously.

\begin{con}
Fix a sheaf $F:X_{\bullet}\rightarrow\mathrm{Vect}_{k}$. For each $n\in\mathbb{N}$, construct a multiset $\mathrm{Dgm}_{n}(X_{\bullet};F)$ over $\mathbb{Z}_{\infty}\times\mathbb{Z}_{\infty}$ called the degree $n$ sheaf persistence diagram associated to $F$ as follows. If $\mathrm{Int}_{n}(i,j)\neq\emptyset$, add a point $(i,j)\in\mathrm{Dgm}_{n}(X_{\bullet};F)$. The multiplicity function $\mu:\mathrm{Dgm}_{n}(X_{\bullet};F)\rightarrow\mathbb{N}$ is given by $\mu(i,j)=|\mathrm{Int}_{n}(i,j)|$.
\end{con}

The multiset $\mathrm{Dgm}_{n}(X_{\bullet};F)$ provides a global description of the lifetimes of sheaf cohomology classes across the filtered space $X_{i}$. For example, clustering of points along the diagonal indicates an instability of sheaf cohomology data when passing to smaller subspaces as dictated by the nature of the filtration. 
\section{Application to Information Flow in Deteriorating Networks}
We will now present an application of our tools to studying information flow across unstable networks. We first recall the basics of network coding sheaves found in [2].
\subsection{Network Coding Sheaves}
Let $F$ be a network viewed as a directed graph. Denote by $V(G)$ and $E(G)$ its sets of nodes and edges, respectively. Assume that there exist sets $S,T\subset V(G)$ called \textit{sources} and \textit{targets}. Define a function $\mathrm{cap}:E(G)\rightarrow\mathbb{N}$ which assigns to each edge $e$ in $F$, its capacity $\mathrm{cap}(e)\in\mathbb{N}$. Let $\mathrm{Vect}_{k}$ denote the category of finite vector spaces and $k$-linear maps for some Galois field $k$. We can construct a cellular sheaf $F:G\rightarrow\mathrm{Vect}_{k}$ called a \textit{network coding sheaf} according to the definition below.
\begin{defn}
A \textit{network coding sheaf} $F:G\rightarrow\mathrm{Vect}_{k}$ is a cellular sheaf constructed as follows. To each edge $e\in E(G)$, $F(e)=k^{\mathrm{cap}(e)}$. For a node $v$, denote by $\mathrm{In}(v)$ the set of directed edges that are directed towards $v$. For a node $v$, $F(v)=\bigoplus_{e\in\mathrm{In}(v)}\mathrm{cap}(e)$. The restriction maps are given by the canonical projections.
\end{defn}
The main result of [2] is the following.
\begin{thm}
Let $F$ be a NC sheaf on directed graph $F$. Then the $0^{th}$ cellular sheaf cohomology vector space $H^{0}(G;F)$ is equivalent to the total information flows on $F$.
\end{thm}
\begin{con}\label{fil}
Let $F$ be a finite directed graph (regarded as a 1-dimensional cell complex) with the structure necessary to construct a NC sheaf $F:G\rightarrow \mathrm{Vect}_{k}$. Let $\mathrm{st}:E(G)\rightarrow\mathbb{R}$ be a function on the edge set of $F$ assigning to each edge $e$ in $F$ its \textit{strength} $\mathrm{st}(e)\in\mathbb{R}_{\geq 0}$. We may constrain the domain to the positive reals, but this is not entirely necessary. Denote by $F_i$ the subgraph of $F$ consisting of all nodes in addition to edges $e$ such that $\mathrm{st}(e)\leq e$. Call positive real $r$ \textit{critical} if $F_r\neq G_{r-\epsilon}$ for some $\epsilon>0$, i.e the subcomplex $F_i$ is strictly larger that $F_{r-\epsilon}$. Denote by $\{c_1, c_2,\ldots\}$ the set of critical values in increasing order. We have a filtration $G_{\bullet}^{\mathrm{st}}=G_{c_{1}}\subset G_{c_{2}}\subset\ldots\subset G_{c_{\infty}}$. 
\end{con}
\begin{rem}
The filtration $G_{\bullet}^{\mathrm{st}}$ is solely indicative of network link strength. The smallest subcomplexes contain the edges with the weakest strength functions.
\end{rem}
\begin{notn}
Fox $X$ a finite directed graph with a strength function $f$, $X_{\bullet}^{f}$ will denote the filtered cell complex in the sense of Construction \ref{fil}.
\end{notn}
\begin{prop}\label{coh}
The co-persistent NC sheaf cohomology vector spaces $H_{i,j}^{0}(X_{\bullet};F)$ are equivalent to the information flows on network $F$ which survive with the removal of edges $E(X_j\backslash X_i)$.
\end{prop}
\begin{proof}
By [2, Theorem 8], the vector spaces $H^0(X_i)$ and $H^0(X_j)$ are equivalent to the information flows across sub-networks $F_i$ and $F_j$, respectively. The image of the map on vector spaces $H^k(X_j;F)\rightarrow H^k(X_i;F)$ induced by the inclusion $F_{i}\hookrightarrow X_j$ is generated by precisely the NC sheaf cohomology classes that are present in $F_j$ but are also present in the sub-network $F_i$. This is the definition of $H_{i,j}^{0}(X_{\bullet};F)$, so the proof is completed.
\end{proof}

One can also gain useful information from the degree 0 persistence diagram $\mathrm{Dgm}_{0}(X_{\bullet};F)$ associated to the NC sheaf $F$ and the filtered network $F_{\bullet}$. 

\begin{con}
Applying the degree $0$ NC sheaf cohomology functor $H^{0}(-;F)$ to the filtered network $X_{\bullet}$, we obtain a persistence module (and therefore a persistence diagram). The decomposition into interval modules indicates which information flows on the network persist through the network's deterioration based on their length. The longest intervals survive through the most edge deterioration and vice versa.
\end{con}
\begin{rem}
The co-persistence of degree $0$ NC sheaf cohomology classes is also indicative of the robustness of certain information flow channels in the network. Indeed, more robust channels will permit more stable information flows across them, which is shown through the length of intervals in the interval decomposition of $H^{0}(X_{\bullet})$.
\end{rem}




\end{document}